\documentclass[12pt,reqno,makeidx]{amsart}
\usepackage{graphicx} %
\usepackage[left=25mm,right=25mm,top=25mm,bottom=25mm]{geometry}

\makeatletter
\let\saved@setaddresses\@setaddresses %
\let\@setaddresses\relax              %
\makeatother

\usepackage{euflag}
\usepackage{amsmath, amssymb}

\usepackage[shortlabels]{enumitem}

\DeclareMathAlphabet{\lcal}{U}{dutchcal}{m}{n}

\allowdisplaybreaks

\usepackage{xcolor} %
\definecolor{darkred}{rgb}{0.5,0,0} %
\definecolor{darkblue}{rgb}{0,0,0.5} %

\usepackage[pagebackref,colorlinks,linkcolor=darkred,citecolor=darkblue,urlcolor=darkblue,hypertexnames=true]{hyperref}

\usepackage{mathrsfs}
\usepackage[normalem]{ulem}

\newcommand{\ccA}{\mathscr{A}}
\newcommand{\ccC}{\mathscr{C}}

\newcommand{\cA}{\mathcal{A}}
\newcommand{\cB}{\mathcal{B}}
\newcommand{\cC}{\mathcal{C}}

\newcommand{\cE}{\mathcal{E}}
\newcommand{\cF}{\mathcal{F}}

\newcommand{\cH}{\mathcal{H}}

\newcommand{\cK}{\mathcal{K}}

\newcommand{\cM}{\mathcal{M}}
\newcommand{\cN}{\mathcal{N}}

\newcommand{\cT}{\mathcal{T}}

\newcommand{\la}{\langle}
\newcommand{\ra}{\rangle}
\newcommand{\II}[1]{\left\langle x \right\rangle_{#1}}

\newcommand{\vg}{{\tt{g}}}
\newcommand{\freeg}{\mathbb F_{\vg}}
\newcommand{\C}{\mathbb C}
\newcommand{\N}{\mathbb N}

\newcommand{\real}{\operatorname{real}}
\newcommand{\imag}{\operatorname{imag}}
\usepackage{imakeidx}
\makeindex

\newcommand{\df}[1]{{\it{#1}}{\index{#1}}}

\newcommand{\ad}{\lcal{d}}
\newcommand{\Ftgd}{\mathcal{F}^2_{\vg,\ad}}  
\newcommand{\AB}{B} 
\newcommand{\ME}{\mathbb{E}}  
\newcommand{\SG}{G}  %

\newtheorem{thm}{Theorem}[section]

\newtheorem{corollary}[thm]{Corollary}
\newtheorem{lemma}[thm]{Lemma}

\newtheorem{prop}[thm]{Proposition}

\newtheorem{remark}[thm]{Remark}

\numberwithin{equation}{section}

\title[Positive operator-valued noncommutative polynomials are squares]{Positive operator-valued noncommutative\\[.5mm] polynomials are %
squares}

\author[A.\ Jindal]{Abhay Jindal${}^{1,Q}$}
\address{Abhay Jindal, Faculty of Mathematics and Physics, University of Ljubljana, Slovenia}
\email{abhay.jindal@fmf.uni-lj.si}
\thanks{${}^1$Supported by the Slovenian Research Agency 
	program P1-0222 and grant J1-50002. AJ thanks
Ecole Polytechnique and Inria for
hospitality during the preparation of this manuscript. }

\author[I.\ Klep]{Igor Klep${}^{2,Q}$}
\address{Igor Klep, Faculty of Mathematics and Physics, University of Ljubljana 
\& Famnit, University of Primorska, Koper 
\& Institute of Mathematics, Physics and Mechanics,
Ljubljana, Slovenia}
\email{igor.klep@fmf.uni-lj.si}
\thanks{${}^2$Supported by the Slovenian Research Agency 
	program P1-0222 and grants J1-50002,
	N1-0217,
J1-60011, J1-50001, J1-3004 and J1-60025. Partially supported by the Fondation de l’Ecole polytechnique as part
of the Gaspard Monge Visiting Professor Program. IK thanks
Ecole Polytechnique and Inria for
hospitality during the preparation of this manuscript.}

\author[S. McCullough]{Scott McCullough}
\address{Scott McCullough, Department of Mathematics\\
	University of Florida\\ Gainesville} %
    
\email{sam@math.ufl.edu}
\thanks{}

\thanks{${}^Q$This work was performed within the project COMPUTE, funded within the QuantERA II 
Programme that has received funding from the EU's H2020 research and innovation programme under the GA No 101017733 {\normalsize\euflag}}



\begin{document}

\begin{abstract}
We establish operator-valued versions of the earlier foundational factorization results for noncommutative polynomials due to Helton (Ann.~Math., 2002) and one of the authors (Linear Alg.~Appl., 2001). Specifically, we show that every positive operator-valued noncommutative polynomial $p$ admits a single-square factorization $p=r^{*}r$. An analogous statement holds for operator-valued noncommutative trigonometric polynomials (i.e., operator-valued elements of a free group algebra).

Our approach follows the now standard sum-of-squares (sos) paradigm but requires new results and constructions tailored to operator coefficients. Assuming a positive $p$ is not sos, Hahn--Banach separation yields a linear functional that is positive on the sos cone and negative on $p$; a Gelfand--Naimark--Segal (GNS) construction then produces a representing tuple $Y$ leading to contradiction since $p$ was assumed positive on $Y$.

The main technical input is a canonical tuple $A$ of self-adjoint operators and, in the unitary case, a canonical tuple $U$ of unitaries, both constructed from the left-regular representation on Fock space. We prove that, up to a universal constant, the norms $\|p(A)\|$ and $\|p(U)\|$ bound the operator norm of any positive semidefinite Gram matrix $G$ representing the sos polynomial $p$. This uniform control is the key input in showing that the cone of (sums of) squares is closed in the product ultraweak topology on the coefficients. A separate approximation argument then produces a separating functional that is continuous for the weak operator topology (WOT). This two-step passage between the ultraweak and WOT topologies constitutes our separation argument and yields the required WOT closedness of the sos cone. With this in hand, the GNS construction associates to such a separating linear functional a finite-rank positive semidefinite noncommutative Hankel matrix and, on its range, produces the desired tuple $Y$.\looseness=-1
\end{abstract}

\maketitle

\begin{list}{}{%
  \setlength{\leftmargin}{1.5cm}%
  \setlength{\rightmargin}{1.5cm}%
  \setlength{\listparindent}{0pt}
  \setlength{\itemsep}{0pt}
  \setlength{\parsep}{0pt}
}
  \item
  \tableofcontents
\end{list}

\section{Introduction}

Positivity and factorization lie at the heart of real algebraic geometry and operator theory. In the commutative setting, positivity certificates via sums of squares (sos) trace back to 
Hilbert's $17^{\rm th}$ problem in 1900; for classical results and modern treatments see \cite{BCR98,Mar08,Sc24}.\looseness=-1

In the 21st century, motivated by developments in linear systems theory \cite{SIG98,dOHMP09}, quantum physics \cite{brunner}, and free probability \cite{MS17}, the free (noncommutative) counterpart has evolved into a broad program within noncommutative function theory \cite{KVV14,MS11,AM15,BMV16,PTD22}. This framework encompasses noncommutative factorizations and noncommutative Positivstellens\"atze. Early landmarks include Helton’s theorem that (scalar) positive noncommutative polynomials are sums of squares \cite{Hel02} and McCullough’s factorization theory for noncommutative polynomials \cite{McC01}; see also \cite{HM04,HMP04,Po95,JM12,JMS21} and the references therein for further developments.

This paper establishes operator-valued analogs of these factorization theorems: every positive operator-valued noncommutative polynomial $p$ admits a single-square factorization $p=r^{*}r$, with an analogous result for operator-valued noncommutative trigonometric polynomials (elements of the free group algebra).

Beyond the noncommutative positivity literature, our results resonate with classical and modern operator factorization themes, including canonical/state-space factorizations of Bart--Gohberg--Kaashoek and collaborators \cite{BGK79,BGKR10}, and the operator Fejér--Riesz and multivariable outer factorization lines \cite{DR10,DW05,GW05}. While our focus is the free (noncommutative) polynomial and free group contexts, the methods developed, such as the WOT-closure mechanism via Fock-space evaluations and the finite-rank Hankel realization, are of independent interest and may be useful in adjacent problems within free analysis and operator theory.

\subsection*{Guide to the introduction.}
Notation is introduced in Subsection~\ref{ssec:notation}. The main results are stated and their proofs outlined in Subsection~\ref{ssec:mainresults}, while Subsection~\ref{ssec:readersguide} provides a roadmap for the remainder of the paper.

\subsection{Notation}\label{ssec:notation}
Fix a positive integer $\vg$. Let 
\df{$\la x \ra$}  denote the free monoid on the $\vg$ letters of the alphabet $x = \{ x_{1}, \dots, x_{\vg}\}.$ 
Its multiplicative identity is the empty word 
\df{$\varnothing$}. 
We endow $\la x\ra$ with the \df{graded lexicographic order}.
The length of a word $w \in \la x \ra$ is denoted by $|w|.$ The
set of all elements (words) of $\la x \ra$ of length (or degree) at most $d$
is denoted \df{$\la x \ra_{d}$}.
The cardinality of $\la x\ra_d$  is \index{$N(d)$}
\[
	N(d) \ =\  \sum_{i=0}^d \vg^i \ =\  \frac{\vg^{d+1}-1}{\vg-1}.
\]

Let \df{$\cH$} be a fixed complex Hilbert space.
Let  \df{$\cB(\cH)$} be the space of all bounded linear operators on $\cH$, and 
let 
\df{$\cA $} to be the free semigroup $\cB(\cH)$-algebra on $x,$
i.e., $\cA=\cB(\cH)\la x\ra = \cB(\cH) \otimes \mathbb{C} \la x \ra$. An element $p$ of $\cA$ takes the form, 
\begin{equation}\label{eq:poly}
 p \ =\  \sum_{w\in\la x\ra}^{\rm finite} P_{w} w,
\end{equation}
where
$P_{w} \in \cB(\cH),$ and is referred to as an (operator-valued) 
\df{polynomial} in $x.$ 
Let \df{$\cA_{d}$} denote the elements from $\cA$ of degree at most $d$.

Equip $\cA$ with the involution \df{$^*$}: on letters, $x_{j}^{*} = x_{j},$ on a word $w = x_{i_1} \cdots x_{i_n} \in \la x \ra,$ 
\[
  w^{*} \ =\  x_{i_n} \cdots x_{i_1};
\]
and, on a polynomial $p$ as in \eqref{eq:poly}, \index{$p^*$}
\[
   p^{*} \ =\  \sum P_{w}^{*} w^{*},
\]
where $P_{w}^{*}$ is the adjoint of the operator $P_{w}$ in $\cB(\cH).$
	
Let $\cK$ be a Hilbert space and $X = (X_{1}, \dots, X_{\vg})$ be a tuple of operators from $\cB(\cK).$  The 
\df{evaluation} of $p$ at $X$ is defined as 
\[
 p(X) \ =\  \sum P_{w} \otimes X^{w},
\]
where $X^{w} = X_{i_1} \dots X_{i_n}$ for $w = x_{i_1} \dots x_{i_n}.$ In general, $p(X)^{*}$ (the adjoint of $p(X)$) and $p^{*}(X)$ are not the same. They are the same if $X$ is a tuple of self-adjoint operators. 

\subsubsection{Trigonometric polynomials}

We will also be interested in evaluating noncommutative polynomials in tuples of unitaries on Hilbert space. An appropriate setting to consider these is the group algebra of the free group $\freeg$ on the $\vg$  letters  $x_i$, $i=1,\ldots,\vg$. 
Elements of $\freeg$ are (reduced) words in the alphabet $x_i,x_i^{-1}$.

Let $\ccA$ be the algebra $\cB(\cH)[\freeg] = \cB(\cH)\otimes\C[\freeg]$. Its elements are called \df{trigonometric polynomials}, and $\ccA$
is
 endowed with the involution
\begin{equation}
\label{d:evalPunitary}
	\sum_{u\in\freeg}^{\rm finite} P_u u \ \mapsto\  \sum_u P_u^* u^{-1}.
\end{equation}
If $X$ is a tuple of unitary operators, then $(X^w)^* = X^{w^{-1}}$ and so 
\[
  p(X)^* \ =\   \left ( \sum P_w \otimes X^w\right)^* \ =\ \sum P_w^* \otimes X^{w^*}   \ =\  p^*(X)
\]
for all $p\in\ccA$. The notions of length of a word, degree of a polynomial, etc.~extend naturally to $\ccA$ and we let,
for positive integers $d,$ \index{$\ccA$} \index{$\ccA_d$}
\[
\ccA_d \ =\ \left\{ 
\sum_{\substack{u\in\freeg\\ |u|\le d}} P_u u \, :\, P_u \in \cB(\cH)
  \right\}.
\]
The number of words in $\freeg$ of length $\le d$, $(\freeg)_d$, is denoted by \df{$N_{\rm red}(d)$} and equals
\[
	N_{\rm red}(d)
	\ =\ 
	1+\sum_{k=1}^{d} 2\vg\,(2\vg-1)^{k-1}
\ =\  \frac{\vg (2\vg-1)^{d}-1}{\vg-1}.
\]
	
\subsection{Main results}\label{ssec:mainresults}
We are now ready to state our main results. The first is an operator-valued version of the classical %
sum of squares theorem of Helton \cite{Hel02} and McCullough \cite{McC01}, Theorem~\ref{thm:sos}. The second, Theorem~\ref{thm:usos},
is a factorization result for positive operator-valued trigonometric polynomials extending a long list of 
results pertaining to scalar-valued noncommutative trigonometric polynomials \cite{McC01,HMP04,BT07,NT13,KVV17,Oz13}. For a bounded operator $T$, the notation $T \succeq 0$ means that the operator $T$ is positive semidefinite (psd).

\begin{thm}\label{thm:sos}
For $f\in\cA_{2d}$ the following are equivalent:
\begin{enumerate}[\rm(i)]\itemsep=5pt
\item \label{i:sos:i}
For any Hilbert space $\cK$ and any tuple of self-adjoint operators $Y = (Y_{1}, \dots, Y_{\vg})\in\cB(\cK)^\vg$, $f(Y)\succeq0$;
\item \label{i:sos:ii}
For any $n\in\N$ and any tuple of self-adjoint matrices 
$Y = (Y_{1}, \dots, Y_{\vg})\in M_n(\C)^\vg$, $f(Y)\succeq0$;
\item \label{i:sos:iii}
There exist $r_1,\ldots,r_{N(d)}\in\cA_d$ s.t.
\begin{equation}\label{eq:sosthm}
	f \ =\ \sum_{i=1}^{N(d)} r_i^*r_i.
\end{equation}
\end{enumerate}

\noindent If $\cH$ is infinite-dimensional, then the above statements are also equivalent to
\begin{enumerate}[label=\textup{(\roman*)}, resume]\itemsep=5pt
\item \label{i:sos:iv}
There exists $r\in\cA_d$ s.t.
\begin{equation}\label{eq:1sosthm}
	f \ =\ r^*r.
\end{equation}

\end{enumerate} 
\end{thm}

\begin{remark} \rm
Several remarks related to Theorem~\ref{thm:sos} are in order. %
\begin{enumerate}[\rm(a)]\itemsep=5pt
\item Item~\ref{i:sos:iii} can also be phrased as a factorization result. 
Letting $r=\text{col} \begin{pmatrix} r_1 & \cdots & r_{N(d)}\end{pmatrix} \in 
\cA^{N(d)}=
\cB(\cH,\cH^{N(d)})\la x\ra$, \eqref{eq:sosthm} simply states
\[
	f\ =\ r^*r.
\]
We refer to \cite{BGK79,BGKR10,DW05,GW05,DR10} and the references therein for an in depth investigation of factorization.
\item
   That item~\ref{i:sos:iii} implies 
  item~\ref{i:sos:i}  implies item~\ref{i:sos:ii}  is trivial. The main content of Theorem~\ref{thm:sos} is that item~\ref{i:sos:ii}
  implies item~\ref{i:sos:iii}. A routine argument shows the equivalence between \eqref{eq:sosthm} and \eqref{eq:1sosthm}  in the infinite-dimensional case, see Remark~\ref{rem:soscone}.
\item
Our proof yields no bound on the size $n$ of matrices needed in item~\ref{i:sos:ii}. 
\item
From Theorem~\ref{thm:sos} one can easily 
deduce its version for free non-self-adjoint variables $z,z^*$ via the usual identification
$z_j\mapsto \real{z_j}= \frac{z_j+z_j^*}2$ and hence
$z_j^*\mapsto \imag{z_j}= \frac{z_j-z_j^*}{2i}$.
\qed
\end{enumerate}
\end{remark} The following result is the unitary version of Theorem~\ref{thm:sos}.

\begin{thm}\label{thm:usos}
For $f\in\ccA_{2d}$ the following are equivalent:
\begin{enumerate}[\rm(i)]\itemsep=5pt
\item \label{i:usos:i}
For any Hilbert space $\cK$ and any tuple of unitary operators  $U = (U_{1}, \dots, U_{\vg})\in\cB(\cK)^\vg$, $f(U)\succeq0$;
\item \label{i:usos:ii}
For any $n\in\N$ and any tuple of unitary matrices 
$U = (U_{1}, \dots, U_{\vg})\in M_n(\C)^\vg$, $f(U)\succeq0$;
\item \label{i:usos:iii}
There exist $r_1,\ldots,r_{N_{\rm red}(d)}\in\ccA_d$ s.t.
\begin{equation*}\label{eq:usosthm}
	f\ =\ \sum_{i=1}^{N_{\rm red}(d)} r_i^*r_i.
\end{equation*}
\end{enumerate}
If $\cH$ is infinite-dimensional, then the above statements are also equivalent to
\begin{enumerate}[\rm(i)]\itemsep=5pt
\item[\rm (iv)] \label{i:usos:iv}
There exists $r\in\ccA_d$ s.t.
\begin{equation*}
	f\ =\ r^*r.
\end{equation*}
\end{enumerate}
\end{thm}

\begin{remark}[What's new?]\rm
\label{rem:whatsnew}
The passage to operator coefficients necessitates several novel
results and constructions that we expect to be of independent interest. At a high level, the proofs of Theorem~\ref{thm:sos} and Theorem~\ref{thm:usos} still follow the now standard paradigm for establishing sum of squares (sos) representations (factorizations).
Namely, the Hahn-Banach theorem produces a separating linear functional $\varphi$, and then a Gelfand-Naimark-Segal (GNS) construction
based on $\varphi$ ultimately produces a tuple $Y$.  
Here we 
roughly follow the outline of \cite{MP05}.  

A key construction is that of a tuple of self-adjoint operators $A$
based upon the left regular representation on Fock space; see Section~\ref{sec:fock}.
We then show that, up to a universal constant, for a sum of squares polynomial $p,$ 
 the norm
of $p(A)$ bounds the norm of any non-commutative psd Gram matrix $G$ that represents $p$; see Proposition~\ref{prop:bounded}.  
This uniform bound is the main input in Section~\ref{sec:top} for proving that the cone $\cC_d$ of sums of squares is closed in the product ultraweak topology on the coefficients. A separate approximation argument then replaces an ultraweak continuous separating functional by a WOT continuous one and hence yields closedness of $\cC_d$ in the product WOT. This two-step interplay between the ultraweak and WOT topologies enables an application of the Hahn-Banach separation theorem.

On the GNS side, we introduce a new argument that exploits the WOT to associate to a
separating linear functional a finite-rank psd noncommutative representing Hankel matrix and,
on its range, construct the desired tuple~$Y$; see Section~\ref{sec:gns}. For the unitary result,
Theorem~\ref{thm:usos}, we additionally modify the construction of Section~\ref{sec:fock} to produce a canonical
tuple of unitary operators from the left-regular representation and adapt the GNS procedure
to obtain a unitary tuple together with a representing vector that realizes the separating
functional as a vector state; see Subsection~\ref{ssec:uGNS}.\qed
\end{remark}

\subsection{Reader's guide}\label{ssec:readersguide}

The paper is structured as follows. The convex cone of sums of squares (making an appearance in \eqref{eq:sosthm} of Theorem~\ref{thm:sos}) is introduced and characterized in the next Section~\ref{sec:sos}. In Section~\ref{sec:fock} we define creation operators $L_i$ on the full Fock $\cF_\vg^2$ space and their symmetrized analogs $A_i$. How they pertain to the sum of squares statement at hand is explored in Subsection~\ref{ssec:coeff}, where evaluations at $A$ are used to extract coefficients of a polynomial. In Section~\ref{sec:top} we 
introduce a suitable topology on $\cA_\ad$ and
collect all the necessary topological properties needed in the sequel. With respect to this topology, the convex cone of sums of squares is closed, see Proposition~\ref{prop:closedcone}. The fact that the cone is closed allows for an application of the Hahn--Banach Separation Theorem, which is then followed by an appropriate version of the GNS construction, carried out in Section~\ref{sec:gns}; see Proposition~\ref{prop:GNS}. 
Then Theorem~\ref{thm:sos} is proved in Section~\ref{sec:proof} and
Theorem~\ref{thm:usos} is proved in Section~\ref{sec:uproof}.

\subsection*{Acknowledgment} The authors thank Dr. Matthias Schötz for carefully reading an earlier version of this manuscript and for pointing out gaps in the proof of Proposition~\ref{prop:closedcone}. 

\section{Convex Cone of (Sums of) Squares}\label{sec:sos}

In this section a key player in the proof of Theorem~\ref{thm:sos}, the convex cone of sums of squares of polynomials, is introduced and
studied. The main result in this section is Proposition~\ref{prop:closed under addition} (see also Remark \ref{rem:soscone}), which  gives a bound on the number of sums of squares needed to write a polynomial as a sum of squares.

\begin{lemma}\label{lem:psd}
If $T: \cH^{n} \to \cH^{n}$ be a psd linear map, then there exist linear maps $R_{i}: \cH \to \cH^{n},$ $i = 1, \dots, n,$ such that $T = \sum_{i=1}^{n} R_{i}R_{i}^{*}.$ Moreover, if $\cH$ is infinite-dimensional, then $T = R R^{*}$ for some $R : \cH \to \cH^{n}.$
\end{lemma}  
\begin{proof}

Since $T$ is psd, there exists a linear map $\tilde{R}: \cH^{n} \to \cH^{n}$ such that $T = \tilde{R} \tilde{R}^{*}.$ Write 
\[
  \tilde{R} \ =\  \begin{bmatrix}
	R_{1}, \dots, R_{n}
\end{bmatrix}
\]
with respect to the orthogonal decomposition $\cH^{n} = \cH \oplus \dots \oplus \cH.$ The first part of the lemma follows by noting that each $R_{i}$ is a map from $\cH$ into $\cH^{n}.$ For the moreover part, let $U: \cH \to \cH^{n}$ be any unitary, and set $R = \tilde{R} U.$
\end{proof}

Index \df{$\mathbb{\cH}^{N(d)}$} and \df{$\cA_{d}^{N(d)}$} (the algebraic direct sum of $\cA_{d}$ with itself $N(d)$ times) by $\la x \ra_{d}.$ 
Let $V_{d} \in \cA_{d}^{N}$ denote the \df{Veronese column vector} whose $w\in \la x \ra_{d}$ entry is $w$
(adopting the  usual convention of viewing $w$ as the the $\cB(\cH)$-valued polynomial $I_\cH \, w$). For instance, 
if $\vg=2$ and $d=2$, then  \index{$V_d$} 
\[
	V_2\ =\  \text{col} \begin{pmatrix}
	1 & x_1 & x_2 & x_1^2 & x_1 x_2 & x_2 x_1 & x_2^2
	\end{pmatrix}.
\]
 Let \df{$\cC_d$} denote the \df{cone of sums of squares} of polynomials of degree at most $d,$
 \begin{equation}\label{eq:sosdef}
 \cC_{d} \ :=\  \left\{ \sum\limits_{i=1}^{N(d)} r_{i}^{*}r_{i}: \quad r_{i}\in \cA_{d},  \quad  i=1, \dots, N(d) \right\}
\subseteq \cA_{2d}.
\end{equation}
Given $r\in \cA_d,$ the column vector  $R$  with $w$ entry $R_{w}^{*}$ 
is called the \df{coefficient vector} of $r$ since  $r = R^{*} V_{d}.$ In particular, 
\[
 r^* r \ =\  V_d^* RR^* V_d
\]
 so that  $r^*r$ has a representation  as  $V_d^* \SG V_d$ for a psd matrix $\SG.$

\begin{prop} \label{prop:closed under addition}
 A polynomial $p\in \cA_{2d}$ is in $\cC_d$ if and only if there is a psd block matrix $\SG$ such that
\begin{equation}
\label{e:gram-rep}
  p \ =\  V_d^* \SG V_d. 
\end{equation}
In fact, if $p=V_d^* \SG V_d,$ then factoring $\SG=\sum_{j=1}^{N(d)} R_jR_j^*$  with $R_j:\cH \to \oplus_{w\in \la x\ra_d}\cH$
as in Lemma~\ref{lem:psd}, 
setting $r_j = R_{j}^{*} V_{d} $
gives,
\[
 p \ =\ \sum_{j=1}^{N(d)} r_j^* r_j.
\]

In particular,  the set $\cC_d$  is a (convex) cone. %
\end{prop}  

We call any psd block matrix $\SG$ satisfying equation~\eqref{e:gram-rep} a \df{Gram representation} for $p.$

\begin{proof}
Given a sum of squares $p =  \sum_{i=1}^{N(d)} r_{i}^{*}r_{i},$ writing $r_j=R_j^{*} V_d$ gives  
$p = \sum_{i=1}^{N(d)} V_{d}^{*} R_{i} R_{i}^{*} V_{d},$ where $R_{i}$ is the coefficient vector corresponding to the polynomial $r_{i}.$  It follows that  $p = V_{d}^{*} \SG V_{d},$ where $\SG = \sum_{i=1}^{N(d)} R_{i} R_{i}^{*}.$ In particular, $\SG : \cH^{N(d)} \to \cH^{N(d)}$ is a psd linear map.

Conversely, suppose there is a psd linear map $\SG:\oplus_{w\in \la x\ra_d} \cH \to \oplus_{w\in \la x \ra_d}\cH$ such that $p = V_d^* \SG V_d.$ 
By Lemma~\ref{lem:psd}, there exist $R_j: \cH\to \oplus_{w\in \la x \ra_d}\cH$ such  that
$\SG=\sum_{j=1}^{N(d)} R_j R_j^*.$ Setting 
$r_j =R_j^* V_d$, one obtains $p=\sum_{j=1}^{N(d)} r_j^*r_j.$

By what has already been proved,
if  $p,q \in \cC_{d},$ then there exist (psd) Gram representations  $p = V^{*} \SG_{p} V$ and $ q = V^{*} \SG_{q} V.$
Now $p + q = V_{d}^{*} (\SG_{p} + \SG _{q}) V_{d} .$ Since $\SG_{p} + \SG_{q} : \cH^{N(d)}  \to \cH^{N(d)}$ is a psd linear map,  
what has already been proved shows 
$p + q \in \cC_{d}.$ 
\end{proof}
	
\begin{corollary}\label{cor:soscone}
Letting   $V_{d} \in \cA_{d}^{N(d)}$ denote  the Veronese column vector,
the convex cone of sums of squares of degree at most $2d$ is
\[
\cC_{d} = \left\{ V_{d}^{*} \SG V_{d}: \quad \SG =  [\SG_{v,w}]_{v,w\in \la x \ra_{d}} \in \cB(\cH)^{N(d) \times N(d)}, \quad \SG \succeq 0 \right\}.
\]
\end{corollary}

\begin{remark}\rm\label{rem:soscone}
It follows from the preceding discussion that the convex cone $\cC_{d}$ takes the form 
\[ 
\cC_{d} \ =\ \{r^{*}r : \quad r \in \cA_{d}\}
\]
when the Hilbert space $\cH$ is infinite-dimensional.
\qed
\end{remark}
	
\section{Full Fock Space and Gram Matrices}\label{sec:fock}

This section recalls the well-known definition of the full Fock space \cite{AP95,JMS21}, the creation operators \cite{Fra84}, and introduces their symmetrized variants in \eqref{eq:symcreate} compressed to a suitable finite-dimensional subspace. In Subsection~\ref{ssec:coeff} we explore how these self-adjoint  operators  are used to extract the coefficients of a polynomial. The main result is Proposition~\ref{prop:bounded} showing that the set of 
positive semidefinite Gram matrices of a polynomial is norm bounded.

The full Fock space can be defined over any Hilbert space. 
The \df{full Fock space} over $\mathbb{C}^{\vg},$ denoted \df{$\cF^2_{\vg}$},  is: 
\[
  \cF^{2}_{\vg} \ =\  \bigoplus_{n=0}^{\infty} (\mathbb{C}^{\vg})^{\otimes n},
\]
where $(\mathbb{C}^{\vg})^{\otimes 0} := \mathbb{C}$ represents the \df{vacuum vector} \df{$\Omega$}.
 Thus elements of $\cF^{2}_{\vg}$ are sequences $(\psi_{0}, \psi_{1}, \psi_{2}, \dots )$ with $\psi_{n} \in (\mathbb{C}^{\vg})^{\otimes n}$ and $\| (\psi_{0}, \psi_{1}, \psi_{2}, \dots ) \|^{2} = \sum\limits_{n=0}^{\infty} \|\psi_{n}\|^{2} < \infty.$
	
\subsection{Basis} 
Let $\{e_{1}, \dots, e_{\vg}\}$ be any orthonormal basis of $\mathbb{C}^{\vg}.$ With any $w = x_{i_{1}} \dots x_{i_{n}} \in \la x\ra,$ associate a vector 
\[
e_{w} \ =\  e_{i_{1}} \otimes \dots \otimes e_{i_{n}} \in (\mathbb{C}^{\vg})^{\otimes n}.
\]
The set $\{e_{w}: w \in \la x\ra \}$ forms an orthonormal basis for $\cF^{2}_{\vg}$, with $e_{\varnothing}$ corresponding to the vacuum vector $\Omega.$
	
\subsection{Left creation operators}
For each $i =1, \dots, \vg,$ define the \df{left creation operator}  \df{$L_i$} on $\cF^{2}_{\vg}$ by 
\begin{equation}\label{eq:create}
 L_{i} (e_{w}) \ =\  e_{x_{i} w} \in (\mathbb{C}^{\vg})^{\otimes (|w|+1)}, \quad (w \in \la x\ra).
 \end{equation}
Clearly, each $L_{i}$ is an isometry. Moreover $L_{i}^{*} L_{j} = 0 $ if $i \neq j.$ Thus the tuple $L = (L_{1}, \dots, L_{\vg})$ 
is a \df{row isometry}. Let \df{$\AB$}
\begin{equation*}
\AB_{i} \ =\  L_{i} + L_{i}^{*}, \quad i =1, \dots, \vg.
\end{equation*}

Fix a positive integer $\ad$. Let $\Ftgd$ denote the subspace of $\cF^{2}_{\vg}$ spanned by 
$\{e_{w}: w \in \la x\ra_{\ad} \}$
and $\iota=\iota_\ad:\Ftgd\to \cF^{2}_{\vg}$ the inclusion. Thus, for instance, for $|w|\le \ad,$
\begin{equation}
    \label{e:Lw}
 \iota^*L_i \iota e_w = \iota^* L_i e_w =\begin{cases} e_{x_i w} & \text{ if } |w|<\ad \\
          0 & \text{ if } |w|=\ad. \end{cases}
\end{equation}
Similarly, if $|v|\le \ad,$ then 
\begin{equation}
    \label{e:Lastv}
 \iota^* L_i^* \iota e_v = \iota^*L_i^* e_v = \begin{cases}  e_u & \text{ if } v=x_i u 
  \\ 0 & \text{otherwise}. \end{cases} 
\end{equation}
Let \df{$A$} be the $\vg$-tuple of operators defined as 
\begin{equation}
\label{eq:symcreate}
 A_i = \iota^* \AB_i \iota = \iota^* \left(L_i + L_i^* \right) \iota.
\end{equation}
 From equations~\eqref{e:Lw} and \eqref{e:Lastv}, $A_i w = \AB_i w$ for $|w|<\ad.$ 
 While $A$ depends on $\ad$, we suppress this dependence in the notation for readability; when $\ad$ is not clear from the context, we write $A = A^{(\ad)}$.

\begin{lemma} \label{lem:faithful}
For any $w \in \la  x \ra_{\ad}$,
\[
A^w \Omega \ =\  e_w + \sum_{|v|<|w|} c_{v,w} e_v 
\]
for some scalars $c_{v,w}$.
\end{lemma}
\begin{proof}
The proof proceeds by  induction on word length.  If $w$ is the empty word, then $A^w \Omega = \Omega = e_\varnothing.$ 
 Assume $n\le \ad$ and the claim holds for all words of length $\le n-1$. Let $w = x_{i_1} \dots x_{i_n} =x_{i_1} \widetilde{w}$  be of length $n.$ Then
\begin{align*}
 A^w \Omega 
 \ =&\  L_{i_{1}} A^{\tilde{w}} \Omega + L_{i_{1}}^{*} A^{\tilde{w}} \Omega,
\end{align*}
 By the induction hypothesis, 
\begin{align*}
A^w \Omega 
\ =&\   L_{i_{1}} \left( e_{\tilde{w}} + \sum\limits_{|v| < n-1} c_{v, \tilde{w}} e_{v}\right) + L_{i_{1}}^{*} \left( e_{\tilde{w}} + \sum\limits_{|v| < n-1} c_{v, \tilde{w}} e_{v}\right) \\
\ =&\   e_{w}  + \left(\sum\limits_{|v| < n-1} c_{v, \tilde{w}} e_{x_{i_{1}}v} \right) + L_{i_{1}}^{*} e_{\tilde{w}} +
\left(\sum\limits_{|v| < n-1} c_{v, \tilde{w}} L_{i_{1}}^{*}e_{v}\right). 
\end{align*}
For any word $u$ of length $k,$ $L_{i_{1}}^{*} e_{u}$ is either zero or is $e_{\tilde{u}},$ for some word $\tilde{u}$ of length $k-1,$
and the proof is complete.
\end{proof}

\begin{lemma} \label{lem:invertible M}
The $N(\ad) \times N(\ad)$ scalar matrix 
\[
\ME_{\ad} \ =\  \begin{bmatrix}
\langle A^{w} \Omega, e_{v} \rangle\end{bmatrix}_{v,w\in\la x\ra_{\ad}}
\]
is invertible.
\end{lemma}

\begin{proof}
Recall that we have endowed $\la x \ra$ with graded lexicographic order. If $\ad\ge |v|$ and  $v>w$, then  
\[
\langle  A^{w} \Omega, e_{v} \rangle \ =\  0 
\]
by Lemma~\ref{lem:faithful}. Hence, $\ME_\ad$ is upper triangular.
 Moreover, each diagonal entry is $1$ by Lemma~\ref{lem:faithful}. Thus, $\ME_{\ad}$ is invertible.
\end{proof}
	
\subsection{Extraction formula for coefficients}\label{ssec:coeff}
Let $q = \sum Q_{w} w \in \cA_{\ad}.$ For $v\in \la x\ra_{\ad}$ 
define the linear functional
$\Omega_{v} : \cB(\Ftgd) \to \mathbb{C}$ 
by
\[
  \Omega_{{v}}(T) = \langle  T \Omega, e_{v} \rangle. 
\]
The operator coefficients \(Q_v\) are obtained from $q(A)$  by solving the linear system
\begin{align*} 
Z_{v}(q) \ :=&\ (\mathrm{id}_{\cB(\cH)} \otimes \Omega_{v}) q(A) \ =\   \sum\limits_{w} Q_{w} \otimes \Omega_{v} (A^{w}) \\
\ =&\   \sum\limits_{w} \langle  A^{w} \Omega , e_{v} \rangle \, Q_{w} \ =\   \sum\limits_{w} [\ME_{\ad}]_{v,w} Q_{w},
\end{align*}
where $[\ME_{\ad}]_{v,w}$ is the $(v,w)$ entry of the matrix $\ME_{\ad}.$ In short, \index{$[\ME_d]$}
\begin{equation}\label{eq:coeff2}
Z(q) \ =\  \ME_{\ad} Q,
\end{equation}
where $Z(q)$ and $Q$ are column vectors with $Z_{v}(q)$ and $Q_{v}$  as the $v^{\rm th}$ entry of $Z$ and $Q,$ respectively. 
Since,  by Lemma~\ref{lem:invertible M}, $\ME_d$ is invertible, 
\begin{equation} \label{eq:coeff}
Q \ =\  \ME_{\ad}^{-1} Z(q).
\end{equation} 
{We refer to $\ME$ as the \df{extraction matrix}, and equation~\eqref{eq:coeff} as the \df{extraction formula} for the coefficients of $q.$ Note 
that this formula depends only upon $q(A);$ that is, the coefficients of $q$ are determined uniquely
by $q(A).$}

It follows from equation~\eqref{eq:coeff}  that  there exists a positive constant $\lambda_{\ad}$ (independent of $q$) such that 
\begin{equation} \label{eq:Coeff bound}
 \|Q_{w}\| \ \leq\  \lambda_{\ad} \, \|q(A)\| \quad \text{ for all $w \in \la x\ra_{\ad}$.}
\end{equation}
	
\begin{prop} \label{prop:bounded}
If  $p \in \cC_{d},$ then the set 
\[
 \Gamma_p \ =\  \{\, \SG \in \cB(\cH)^{N(d)\times N(d)} \ : \quad \SG \succeq 0, \quad   V_d^* \SG V_d = p \,\}
\]
 is norm bounded $($with respect to the operator norm on $\cB(\cH^{N(d)})$$)$. More precisely, 
 there exists a constant $\mu_{d}$ $($depending only on $d$ and $\vg$ and not on $p$$)$ such that,
 for all $\SG\in\Gamma_p,$ 
\[ 
\|\SG\| \ \leq\  \mu_{d} \; \|p(A)\| ,
\]
where the tuple $A=A^{(\ad)}$ is defined in \eqref{eq:symcreate}
for any $\ad\geq 2d$. 
\end{prop}
	
\begin{proof}
Fix $p \in \cC_{d}$ and $\SG \in \Gamma_p.$ Thus $p=V_d^* \SG V_d.$ By
Proposition~\ref{prop:closed under addition}, there exists $Q_j: \cH\to \oplus_{w\in\la x\ra_d} \cH$ such that
\begin{equation}
\label{e:bound:1}
p \ =\  \sum\limits_{j=1}^{N(d)} q_{j}^{*}q_{j} \ =\  V_d^* \left [ \sum\limits_{j=1}^{N(d)} Q_{j} Q_{j}^* \right ] V_d,
\end{equation}
where
\[
 q_j \ =\   Q_j^* V_d = \sum_{w\in\la x\ra_d} Q_{j,w} w.
\]
 By equation~\eqref{eq:Coeff bound}, for $v\in\la x \ra_d,$
\[
 \| Q_{j,v}\| \ \le\  \lambda_d\; \|q_j(A)\|.
\]
From equation~\eqref{e:bound:1}, 
\[
  \|q_j(A)\|^2 \ =\  \|q_j(A)^* q_j(A)\| \ \le\  \|p(A)\|.
\]
Thus, again using equation~\eqref{e:bound:1},
\[
  \sum_{u,v\in \la x \ra_d} \|\SG_{u,v}\| \ \le\   \sum_{u,v\in \la x \ra_d} \, \sum_{j=1}^{N(d)} \|Q_{j,u}Q_{j,v}^*\| 
  \ \le\  N(d)^3 \lambda_d^{2} \; \|p(A)\|.
\]
It follows that 
$\|\SG\|\le \mu_d \, \|p(A)\|$ for $\mu_d = N(d)^3 \lambda_d^{2}.$
\end{proof}
	
\section{Topology on \texorpdfstring{$\cA_{\ad}$}{Ad}}\label{sec:top}
The main purpose of this section is to define a well-behaved topology on $\cA_{\ad}$ in which the convex cone of sums of squares 
is closed.

To each polynomial in  $\cA_{\ad}$ we associate the vector of its coefficients as an element in $\cB(\cH)^{\la x\ra_{\ad}}.$
The topology on $\cA_{\ad}$ is then the topology induced from the product WOT on $\cB(\cH)^{{\la x\ra_{\ad}}}.$ 
Alternately, in this topology a net  $(p_{\alpha} = \sum_{w} P_{\alpha, w} w)_{\alpha}$ in $\cA_{\ad}$  converges to $p = \sum_{w} P_{w} w\in \cA_{\ad}$ if {and only if} for each $w \in \la x \ra_{\ad}$, the net of operators  $(P_{\alpha,w})_\alpha$ converges to $P_w$ in WOT. Note that $\cA_{\ad}$ becomes a locally convex topological vector space with this topology.  We refer to this  topology as the WOT on $\cA_{\ad}$ or the \df{product WOT topology}. It is the default topology on $\cA_{\ad};$ that is, unless otherwise stated, it is the topology on $\cA_{\ad}.$

Exclusive to this section, we will also use a product ultraweak topology on $\cA_{\ad}.$
The ultraweak topology on $\cB(\cH)$ is the weak-$*$ topology on $\cB(\cH)$ induced by the
predual $\cT(\cH)$, the trace class operators on $\cH.$
It is the weakest topology such that predual elements remain continuous on $\cB(\cH).$ 
As before, to each polynomial in  $\cA_{\ad}$, we associate a vector of its coefficients as an element in $\cB(\cH)^{\la x\ra_\ad}.$
The ultraweak topology on $\cA_\ad$ is then the topology induced from the product ultraweak topology on $\cB(\cH)^{{\la x\ra_\ad}}.$ Alternately, in this topology a net  $(p_{\alpha} = \sum_{w} P_{\alpha, w} w)_{\alpha}$ in $\cA_{\ad}$  converges to $p = \sum_{w} P_{w} w\in \cA_{\ad}$ if {and only if} for each $w \in \la x \ra_{\ad}$, the net of operators  $(P_{\alpha,w})_\alpha$ converges to $P_w$ in ultraweak topology for all $w \in \la x \ra_{\ad}$. Note that $\cA_{\ad}$ becomes a locally convex topological vector space with this topology.

\subsection{Closedness of the cone \texorpdfstring{$\cC_d$}{Cd}}
The main goal of this subsection is to establish the following result. In the rest of this section fix an integer $\ad\geq 2d$, and set $A=A^{(\ad)}$.

\begin{prop}\label{prop:closedcone}
The convex cone $\cC_{d}$ is closed in $\cA_{\ad}$.
\end{prop}

Firstly, since $\cA_{2d}$ is closed in $\cA_{\ad}$, it suffices to show $\cC_d$ is closed in $\cA_{2d}$ (in its product WOT topology).
The proof of Proposition~\ref{prop:closedcone} proceeds in two steps. We first show that the cone $\cC_{d}$ is closed with respect to the stronger  ultraweak topology.
Next, we prove that if a polynomial in $\cA_{2d}$ can be separated from the cone $\cC_{d}$ by an ultraweak continuous linear functional, then it can also be separated by a (possibly different) WOT continuous linear functional. This implication will yield the desired closedness of $\cC_{d}$.

\subsubsection{Closedness of \texorpdfstring{$\cC_d$}{Cd} in the product ultraweak topology} Let $\cT(\cH)$ be the space of trace-class operators on $\cH.$ The space $\cT(\cH)^{\la x \ra_{\ad}}$ with the norm 
\[
\|(T_{1}, \dots, T_{N(\ad)})\|_{1} \ =\ \sum\limits_{i=1}^{N(\ad)} \|T_{i}\|_{1}, 
\]
is a Banach space, where $\|\cdot\|_{1}$ denotes the trace-class norm of $\cT(\cH)$. The dual of $\cT(\cH)^{\la x \ra_{\ad}}$ is $\cB(\cH)^{\la x \ra_{\ad}}$ with the induced norm 
\[
\|(T_{1}, \dots, T_{N(\ad)})\| \ =\ \max\limits_{1 \leq i \leq N(\ad)}\, \|T_{i}\|,
\]
where $\|\cdot\|$ denotes the operator norm of $\cB(\cH).$ Via the canonical identification, we can endow $\cA_{\ad}$ with a norm. For $p = \sum\limits_{w \in \la x \ra_{\ad}} P_{w}w$, we define
\[
\| p\| \ :=\ \max\limits_{w \in \la x \ra_{\ad}}\, \|P_{w}\|. 
\]

\begin{lemma}
For any $t  > 0$, the truncated convex set 
\[
\cC_{d,t} \ :=\ \{ p \in \cC_{d} \ :\ \|p\| \leq t \}
\]
is closed in the ultraweak topology on $\cA_{\ad}.$ 
\end{lemma}
\begin{proof}
Fix $t >0.$ Let $(p_{\alpha})_\alpha$ be a net in $\cC_{d,t}$ that converges to $p \in \cA_{2d}.$ Our aim is to show that $p \in \cC_{d,t}.$

For each $\alpha$,
\[
\|p_{\alpha}(A) \| \ \leq \  \|p_{\alpha}\|\, \sum\limits_{w\in \cA_{2d}} \|A^{w}\|\ \leq \  t \,\sum\limits_{w\in \cA_{2d}} \|A^{w}\|.
\]
Hence,
\begin{equation}
    \label{eq:paAfinite}
\sup\limits_{\alpha} \|p_{\alpha} (A) \| \ <\  \infty. 
\end{equation}
Choose $G_\alpha\in \Gamma_{p_\alpha}.$ Combining equation~\eqref{eq:paAfinite} and  Proposition~\ref{prop:bounded} gives, 
\[
   \sup\limits_{\alpha} \|\SG_{\alpha}\| \ <\  \infty.
 \]
Thus, by the Banach--Alaoglu theorem, there is a subnet of operators $(\SG_{\beta})_{\beta}$ 
 that  converges to some operator $\SG \in \cB(\cH)^{N(d) \times N(d)}$ in the ultraweak topology. Since the convergence in the ultraweak operator topology is stronger than convergence in the weak operator topology,
 $\SG_{\beta} \longrightarrow \SG$ in the WOT. Thus, $\SG \succeq 0.$ 

Let $q=V_{d}^{*} \SG V_{d} \in \cC_{d}$ and 
observe,
\[
p_{\alpha}(A) \ =\ V_{d}(A)^{*} (\SG_{\alpha} \otimes I_{\cF^{2}_{\vg}}) V_{d}(A), 
\]
 where $V_{d}(A):\  \cH \otimes \Ftgd \to \cH^{N(d)} \otimes \Ftgd,$ is defined by 
\[
 V_d(A) h \otimes \xi = (h \otimes A^{w} \xi)_{w}.
\]
Since $(\SG_{\beta})_\beta$ 
converges to $\SG$ in the WOT and the tuple $A$ acts on a finite dimensional Hilbert space, 
\[
V_{d}(A)^{*} (\SG_{\beta} \otimes I_{\cF^{2}_{\vg}}) V_{d}(A) \  \longrightarrow \  V_{d}(A)^{*} (\SG \otimes I_{\cF^{2}_{\vg}}) V_{d}(A) \ =\  q (A)
\]
in the WOT. 
Thus, $(p_{\beta} (A))_\beta$ converges to $q(A)$ in the WOT. Therefore, $p(A) = q(A).$ Since, by the extraction formula, equation~\eqref{eq:coeff},
 the coefficients of $p$ and $q$ are determined by evaluation at $A,$ it follows that  $p=q.$ Hence, $p \in \cC_{d}.$ Since the norm is lower semicontinuous in the ultraweak topology,  $\|p\| \leq t. $ Hence, $p \in \cC_{d,t}.$
\end{proof}

It follows from Krein--Smulian Theorem (see, e.g., \cite[Theorem 3.6.2]{D25}) that the cone $\cC_{d}$ is closed in the ultraweak topology. 

\subsubsection{Closedness of $\cC_d$ in the product WOT topology} 

\begin{lemma}\label{lem:approx_1}
Given trace-class operators $S_i\in\cT(\cH)$, $1\leq i\leq m$ and $\epsilon>0$, there exists a finite-rank operator $P$ such that
\[
	\|S_i - PS_iP\|_1<\epsilon
\]
for all $i$. Here $\|\cdot\|_1$ denotes the trace norm.
\end{lemma}

\begin{proof}
Since finite-rank operators are dense in $\cT(\cH)$ with respect to the trace norm, there are finite-rank operators $F_i$ with 
\[
\|S_i-F_i\|_1<\frac{\epsilon}2.
\]

Let $P$ be the orthogonal projection onto $\sum_i {\rm ran}(F_i)+ \sum_i {\rm ran}(F_i^*)$. Then {$PF_iP=F_i$} for all $i$. Hence
\[
	\begin{split}
\|S_i-PS_iP\|_1 & \leq \|S_i-F_i\|_1 + \|F_i-P S_i P\|_1 \\
 & = \|S_i-F_i\|_1 + \|PF_iP-P S_i P\|_1 \\
 & = \|S_i-F_i\|_1 + \|P (F_i -S_i)P\|_1 \\
 & \leq \frac{\epsilon}2 + \|P\|\, \|F_i-S_i\|_1 \, \|P\| \\
&\leq \frac{\epsilon}2+ \frac{\epsilon}2=\epsilon. \qedhere
	\end{split}
\]
\end{proof}

\begin{lemma}\label{lem:uweak2WOT}
 If  $\varphi:\mathcal A_{2d}\to\mathbb C$ be an ultraweak continuous linear functional that 
separates the cone $\cC_{d}$ from a fixed polynomial $p$ in $\cA_{2d},$ that is, 
\[
\varphi(r^{*}r) \geq 0 \quad \text{for all $r \in \cA_{d}$} \quad  \text{and} \quad \varphi(p+p^{*}) <0,
\]
then there exists a WOT continuous linear functionl $\tilde{\varphi} : \cA_{2d} \to \mathbb{C}$ such that 
\[
\tilde{\varphi}(r^{*}r) \geq 0 \quad \text{for all $r \in \cA_{d}$} \quad  \text{and} \quad \tilde{\varphi}(p+p^{*}) <0.
\]
\end{lemma}
\begin{proof}
Since $\varphi$ is ultraweak continuous, there exist trace class operators $S_{w}$ $(w \in \la x\ra_{2d})$ in $\cB(\cH)$ such that 
\[
\varphi(q) = \sum\limits_{w \in \la x\ra_{2d}} {\rm Tr}\, (S_{w} Q_{w}),
\]
where $q = \sum_{w \in \la x \ra_{2d}} Q_{w} w.$ For $r,r^\prime \in \cA_{d}$, 
\[
\varphi(r^{*}r^\prime) \ =\ \sum\limits_{u,v\in \la x \ra_{d}} {\rm Tr}\, (S_{u^{*}v} R_{u}^{*}R^\prime_{v}),
\]
where  $r = \sum_{u \in \la x \ra_{d}} R_{u} u$
and $r^\prime = \sum_{v \in \la x \ra_{d}} R^\prime_{v} v.$ Denote by $S$ the $N(d) \times N(d)$ block operator matrix whose $(u,v)$ entry is $S_{v^{*}u}.$

\medskip

For any $r \in \cA_{d}$, define the row operator $R : \bigoplus_{u \in \la x \ra_{d}} \cH \to \cH$ by 
\[
 R (\oplus_{u\in\la x\ra_{d}} \zeta_w) \ =\  \sum_{u} R_u \zeta_u.
\]
Thus, 
\begin{align*} 
\varphi(r^{*}r) 
 \ &=\ \sum\limits_{u,v\in \la x \ra_{d}} {\rm Tr}\, (S_{u^{*}v} R_{u}^{*}R_{v}) \\
 \ &=\  \sum_{u,v\in \la x \ra_{d}} {\rm Tr}\, \left([S]_{v,u} [R^{*}R]_{u,v}\right) \\
 \ &=\  \ \ \ {\rm Tr}\, (SR^{*}R).  
\end{align*}
We claim that
\[
{\rm Tr}\, (S T ) \ \geq\  0
\]
for any positive operator $T \in \cB (\bigoplus_{\la x\ra_{d}} \cH).$ It follows from Lemma~\ref{lem:psd} that $T = R^{*}R$ for some $R: \bigoplus_{\la x\ra_{d}} \cH \to \cH.$ Letting $r = \sum R_{u} u,$ where $R_{u}$ is the $u^{\rm th}$ element of the row operator $R$ gives
\[ 
{\rm Tr}\, (ST) \ =\  {\rm Tr}\, (SR^*R) \ =\  \varphi (r^{*}r) \ \geq\  0.
\]

To prove that $S$ is a positive operator, it remains to show that $S$ is self-adjoint.
For $T \succeq 0,$ ${\rm Tr}\, (ST)$ is real. Hence, 
\[
{\rm Tr}\, (S^{*}T) \ =\  {\rm Tr}\, (T S^{*}) \ =\  {\rm Tr}\, ((ST)^{*}) \ =\  \overline{{\rm Tr}\, (ST) } \ =\  {\rm Tr}\, (ST) . 
\]
Since every bounded operator on a Hilbert space is a linear combination of four positive operators,
\[
{\rm Tr}\, (S^{*}T) \ =\  {\rm Tr}\,(ST) \quad \text{for all $T \in \cB(\bigoplus_{\la x\ra_{d}} \cH)$ }.
\]
Hence $S^{*} = S$ and $S\succeq 0.$ 

The finitely many trace class operators $S_w$, $w\in\la x\ra_{2d},$ can be approximated by finite rank operators in the trace norm
as in {Lemma~\ref{lem:approx_1}}. That is, for any $n \in \mathbb{N}$, there exists a finite-rank projection $P_{n}$ of $\cH$ such that 
\[
\| S_w - P_nS_{w}P_n \|_1 < \frac1n ,
\]
for all $w\in\la x\ra_{2d},$
where $\|\cdot\|_1$ denotes the trace norm. 
Letting  $S^{(n)}$ denote the $N(d) \times N(d)$  block operator matrix whose $(u,v)$ entry is $P_nS_{v^{*}u}P_n,$ 
\[
S^{(n)} \ =\ (I_{{N(d)}} \otimes P_{n}) \, S \,  (I_{{N(d)}} \otimes P_{n}).
\]
Whence $S^{(n)}$ is a finite-rank psd operator. 

Define a linear functional $\varphi_{n} : \cA_{2d} \to \mathbb{C}$ by 
\[
	\varphi_n(Q u^*v) = {\rm Tr}\, (P_{n} S_{u^{*}v} P_{n} \ Q)  
\]
for $Q\in \cB(\cH)$ and $u,v\in\la x\ra_d$. 
For  $r= \sum_{u \in \la x \ra_{d}} R_{u} u$ 
and $r^\prime = \sum_{v \in \la x \ra_{d}} R^\prime_{v} v,$ 
\[
\begin{split}
	\varphi_n(r^*r^\prime) & = \sum_{u,v\in\la x\ra_d} \varphi_n(R_u^*R^\prime_v u^*v)= \sum_{u,v\in\la x\ra_d} {\rm Tr}\, (P_{n} S_{u^{*}v} P_{n} \ R_u^*R^\prime_v)  \\
	& = {\rm Tr}\, (P_{n} S P_{n} \ R^*R^\prime)  = {\rm Tr}\, (S^{(n)} \ R^*R^\prime).
	\end{split}
\]
Since $S^{(n)}$ is psd,
\[
\varphi_{n} (r^{*}r) \ =\ {\rm Tr}\, (S^{(n)} R^{*}R)  \geq0 
\]
for all $r \in \cA_{d}.$ 
Further, since $S^{(n)}$ is a finite-rank operator, $\varphi_{n}$ is WOT continuous. 

Since $(P_{n}S_{u^{*}v} P_{n})_n$ converges to $S_{u^{*}v}$ in 
the trace norm, ${\rm Tr}(P_nS_{u^*v}P_n Q)$ converges
to ${\rm Tr}(SQ).$ Hence $\varphi_{n}(p+p^{*})$  converges to $\varphi(p+p^{*}).$
As $\varphi(p+p^*)<0$, there exists a natural number $n$ such that $\varphi_{n}(p+p^{*}) <0.$ The linear functional $\tilde{\varphi} := \varphi_{n}$ has the desired separation properties.  
\end{proof}

We are finally ready to prove Proposition~\ref{prop:closedcone}.

\begin{proof}[Proof of Proposition~\ref{prop:closedcone}]
Suppose $p\in\cA_{2d}$ is not in $\cC_d$. Since $\cC_d$ is ultraweak closed, there is an ultraweak continuous linear functional $\varphi$ on $\cA_{2d}$ that separates $p$ from $\cC_d$. Now apply Lemma~\ref{lem:uweak2WOT} to obtain a WOT continuous separating linear functional $\tilde\varphi$. Thus $p$ is not in the WOT closure of $\cC_d$. 
\end{proof}

%
%
%

\section{GNS Construction}\label{sec:gns}
	
In preparation of the application of the Hahn-Banach-convex separation theorem in the proof of Theorem~\ref{thm:sos} in Section~\ref{sec:proof}, we establish a suitable version of the GNS construction in Proposition~\ref{prop:GNS}. Before doing so, for the reader's convenience, 
we state three  well-known lemmas that will be used in the proof. 


\begin{lemma} \label{lem:clf1}
Let $\cH$ be a Hilbert space. If $f : \cB(\cH) \to \mathbb{C}$ is a WOT continuous  linear functional, then there exist a finite index set $J$ and vectors $h_{j}, k_{j} \in \cH$ such that
\[
f(T) = \sum_{j \in J}  \langle T h_j, k_j \rangle_{\cH}
\quad \text{for all } T \in \cB(\cH).
\]
\end{lemma}
	
For a proof see \cite[Section 3.1 and Exercise 3.4]{D25}.
	
\begin{lemma} \label{lem:clf2}
If $\varphi: \cA_{\ad} \to \mathbb{C}$ is a continuous linear functional,
then there exist a finite index set $J,$ vectors $h_{j}, k_{j} \in \cH$ and scalars $c_{w} \in \mathbb{C}$ $(w\in \la x\ra_{\ad})$ such that for every $p = \sum P_{w} w \in \cA_{\ad},$
\[
\varphi(p) = \sum\limits_{j \in J} \sum\limits_{w \in \la x \ra_{d}} c_{w} \la P_{w} h_{j}, k_{j} \ra.
\]
\end{lemma}

\begin{proof}
First identify $\cA_{\ad}$ with $\cB(\cH)^{N(\ad)}$ via 
\[
\sum P_{w} w \mapsto (P_{w})_{w}.
\]
This identification induces the product weak operator topology on $\cB(\cH)^{N(\ad)},$ i.e., the topology generated by the seminorms $\| \cdot \|_{w,h,k}$ given by 
\[
\| P \|_{w,h,k} \ :=\  | \la P_{w} h,k \ra |, \quad w \in \la x \ra_{\ad},\   h,k \in \cH,
\]
where $P = (P_{w})_{w}$ is a tuple of $N(\ad)$ many operators in $\cB(\cH).$ Now the statement follows from Lemma~\ref{lem:clf1}.
\end{proof}

The third lemma presents a few basic properties of the vectorization map 
used in the proof of Proposition~\ref{prop:GNS}. See  Subsection~\ref{ssec:vec}.
Let $HS(\cH,\cE)$ denote the Hilbert-Schmidt operators from the Hilbert space $\cH$
to the Hilbert space $\cE.$  For a fixed orthonormal basis $(e_\delta)_\delta$ of $\cH$,
the \df{vectorization map}, 
$\mathrm{vec}:HS(\cH,\cE)\to \cH\otimes \cE$ is defined, for $T\in HS(\cH,\cE),$ by 
\[
\mathrm{vec}(T)\ =\ \sum_\delta e_\delta\otimes Te_\delta 
\]

\begin{lemma}\label{lem:vec identity}
Let $A,B\in {HS}(\cH,\cE)$ and $P,Q\in \cB(\cH)$, $T\in \cB(\cE)$.
\begin{enumerate}[\rm(1)]
\item \label{i:vec:1}
Independent of the choice of orthonormal basis,
\[
\langle \mathrm{vec}(A),\ \mathrm{vec}(B)\rangle_{\cH\otimes \cE}
\ =\ \mathrm{Tr}(A^{*}B). 
\]
In particular, $\|\mathrm{vec}(A)\|=\|A\|_{\mathrm{HS}}.$
\item  \label{i:vec:2}
The following compatibility with left/right actions holds:
\[
(P\otimes I_\cE)\,\mathrm{vec}(A)=\mathrm{vec}(A P^{*}),\qquad (I_\cH\otimes T)\,\mathrm{vec}(A)=\mathrm{vec}(T A), 
\]
hence, more generally,
\[
(P\otimes T)\,\mathrm{vec}(A)=\mathrm{vec}(T\,A\,P^{*}). 
\]
\item \label{i:vec:3}
(Inner–product identity used in Proposition~\ref{prop:GNS})
\[
\big\langle (P\otimes I_\cE)\,\mathrm{vec}(A),\ (Q\otimes I_\cE)\,\mathrm{vec}(B)\big\rangle
\ =\ \mathrm{Tr}\!\big(PA^{*}B\,Q^{*}\big). 
\]
\end{enumerate}
\end{lemma}
	
\begin{proof}
Fix an orthonormal basis $(e_\delta)_\delta$ of $\cH$.

\medskip
		
Using the definition of $\mathrm{vec},$
\[
\langle \mathrm{vec}(A),\mathrm{vec}(B)\rangle
		=\sum_\delta \langle A e_\delta,\ B e_\delta\rangle_\cE
		=\sum_\delta \langle e_\delta,\ A^{*}B e_\delta\rangle_\cH
		=\mathrm{Tr}(A^{*}B). 
\]        
The right-hand side is independent of the choice of orthonormal basis, hence so is the left-hand side,
proving item~\ref{i:vec:1}.
		
\medskip
		
To prove the first identity of item~\ref{i:vec:2}, observe
\[
(P\otimes I_\cE)\mathrm{vec}(A)
		=\sum_\delta P e_\delta \otimes A e_\delta
		=\sum_{\delta,\delta'}\langle e_\delta',Pe_\delta\rangle\, e_\delta'\otimes A e_\delta
		=\sum_\delta e_\delta\otimes A P^{*} e_\delta
		=\mathrm{vec}(A P^{*}). 
\]
The second identity is immediate:
$(I_\cH\otimes T)\sum_\delta e_\delta \otimes A e_\delta =\sum_\delta e_\delta\otimes T(A e_\delta)=\mathrm{vec}(TA)$.
Combine both identities to get $(P\otimes T)\mathrm{vec}(A)=\mathrm{vec}(T A P^{*})$.
		
\medskip
		
Turning to item~\ref{i:vec:3}, using item~\ref{i:vec:2} twice and then item~\ref{i:vec:1},
\[
\begin{split}
\big\langle (P\otimes I)\mathrm{vec}(A),\ (Q\otimes I)\mathrm{vec}(B)\big\rangle
		& =\big\langle \mathrm{vec}(A P^{*}),\ \mathrm{vec}(B Q^{*})\big\rangle \\
		&=\mathrm{Tr}\!\big((A P^{*})^{*} B Q^{*}\big) \\
		&=\mathrm{Tr}\!\big(P A^{*} B Q^{*}\big).  \qedhere
		\end{split}
 \]       
\end{proof}

We have now reached  the main result of this section.

\begin{prop}\label{prop:GNS}
If $\varphi:\mathcal A_{2d+2}\to\mathbb C$ is a continuous linear functional such that
\[
\varphi(p^ *p)\ \ge\ 0  %
\]
for all $p\in \cA_{d+1},$ 
then there exist a finite-dimensional Hilbert space $\cE$, a self-adjoint $g$-tuple $Y=(Y_1,\dots,Y_\vg)$ on $\cE$, and a vector $\gamma\in \cH\otimes \cE$ such that
\[
\varphi(q^ *p)\ =\ \big\langle p(Y)\gamma,\ q(Y)\gamma\big\rangle_{\cH\otimes \cE} \qquad\text{for all }p \in \cA_{d+1},\ q\in\mathcal A_{d}. 
\]
Therefore, for all $p\in\mathcal A_{2d+1}$,
\[
\varphi(p)\ =\ \langle p(Y)\gamma,\ \gamma\rangle. 
\]
\end{prop}

The proof proceeds in five steps.

\subsection{The positive block matrix \texorpdfstring{$S$}{S}}
We construct a finite-rank psd block operator matrix $S$ that determines the linear functional $\varphi.$

 By Lemma~\ref{lem:clf2}, there exist a finite index set $J$, vectors $h_j,k_j\in\cH$, and scalars $c_{u}\in\mathbb C$ ($u\in \II{2d+2}$) such that for every $p=\sum_{u\in \II{2d+2}}P_u\,u$,
\[
\varphi(p)\ =\ \sum_{j\in J}\ \sum_{u\in \II{2d+2}} c_{u}\,\langle P_u h_j,k_j\rangle. 
\]
For $v,w\in \II{d+1}$ define the finite rank operator
\[	
S_{v^ *w}\ :=\ c_{v^ *w}\sum_{j\in J} |h_j\rangle\!\langle k_j|.
\]
Denote by $S$ the $N(d+1) \times N(d+1)$ block operator matrix whose $(v,w)$ entry is $S_{w^{*}v}.$ For $p=\sum_{w\in \la x\ra_{d+1}}P_w\,w$ and $q=\sum_{v\in \la x\ra_{d+1}}Q_v\,v$, a direct computation
gives the block–trace identity
\begin{align}
\varphi(q^*p)
\ =&\  \sum_{j\in J}\sum_{v,w\in \la x \ra_{d+1}} c_{v^*w} \langle Q_v^* P_w h_j, k_j\rangle \nonumber \\
\ =&\  \sum_{v,w\in \la x\ra_{d+1}} c_{v^*w} \sum_{j\in J}  \langle Q_v^* P_w h_j, k_j\rangle \nonumber \\
\ =&\  \sum_{v,w\in \la x \ra_{d+1}} \operatorname{Tr}_\cH ( S_{v^* w}\, Q_v^* P_w). \label{eq:block-trace}
\end{align}
Define the row operator $P : \bigoplus_{\la x\ra_{d+1}}\cH \to \cH$ by 
\[
 P (\oplus_{w\in\la x\ra_{d+1}} \zeta_w) \ =\  \sum_{w} P_w \zeta_w.
\]
The adjoint $P^{*}: \cH \to \cH^{N(d+1)}$ of $P$ is given by %
\[
 P^* \eta  \ =\  \oplus_{w\in \la x\ra_{d+1}} P_w^* \eta.
\]
Now
\begin{align*} 
\varphi(p^{*}p) 
 \ & \ = \sum_{v,w\in \la x\ra_{d+1}} \operatorname{Tr}_\cH ( S_{v^* w}\, P_v^* P_w)   
\\[3pt] \ & \  = \sum_{v,w\in \la x \ra_{d+1}} \operatorname{Tr}_\cH \left([S]_{w,v} [P^{*}P]_{v,w}\right) 
\\[3pt] \ & \  =  \ \ \operatorname{Tr} (SP^{*}P).  
\end{align*}

We claim that
\begin{equation} \label{eq:trace3}
\operatorname{Tr} (S T ) \ \geq\  0 
\end{equation}
for any positive operator $T \in \cB (\cH^{N(d+1)}).$ It follows from Lemma~\ref{lem:psd} that $T = P^{*}P$ (when $\cH$ is infinite-dimensional) for some $P: \cH^{N(d+1)} \to \cH.$ Letting $p = \sum P_{w} w,$ where $P_{w}$ is the $w^{\rm th}$ element of the row operator $P$ gives,
\[ 
\operatorname{Tr} (ST) \ =\  \varphi (p^{*}p) \ \geq\  0.
\]
In case of finite dimensional $\cH,$ a similar argument gives \eqref{eq:trace3}. The only difference is that instead of one polynomial, 
$N(d+1)$ many polynomials as per Lemma~\ref{lem:psd} are required. In either case it follows that $\operatorname{Tr} (ST)\ge 0$ for all
positive operators $T.$

The positivity of $S$ can be concluded exactly as in the proof of Lemma~\ref{lem:uweak2WOT}. Because $J$ and $\la x\ra_{d+1}$ are finite, $S$ is finite rank.
		
\medskip
		
\subsection{An Auxiliary Hilbert space \texorpdfstring{$\cM$}{M}} We  construct a finite-dimensional Hilbert space $\cM$ from the psd sesquilinear form induced by the psd operator $S.$ The Hilbert space $\cE$ is constructed as a subspace of $\cM$ in  Subsection~\ref{ssec:E}.
		
Consider the vector space 
\[
V =\ \bigoplus_{w\in \la x \ra_{d+1}} \cH. 
\]
Equip $V$ with the sesquilinear form 
\begin{align*}	
\left\langle (\xi_w)_{w},\,(\eta_v)_{v}\right\rangle_{V}\  := &\ \la S (\xi_{w})_{w}, (\eta_{v})_{v} \ra_{\cH^{N(d+1)}} 
\ =\  \sum\limits_{v \in \la x\ra_{d+1}} \left \la \sum\limits_{w\in \la x \ra_{d+1}} [S]_{v,w} \xi_{w}, \eta_{v}\right\ra_{\cH} \\
\ =&\  \sum\limits_{v,w \in \la x \ra_{d+1}} \la S_{w^{*}v} \xi _{w}, \eta_{v}\ra_{\cH} 
\ =\  \sum_{v,w\in \la x \ra_{d+1}}\!\langle \xi_w,\ S_{v^ *w}\,\eta_v\rangle_\cH . 
\end{align*}
This form is psd by the positivity of $S.$ Let 
\[
\cN\ =\ \{z\in V:\langle z,z \rangle_{V}=0\}
\]
denote its subspace of null vectors and set
\[
\cM \ =\  V/\cN. \]
Since $S$ has finite rank, $\cM$ is finite dimensional.

\subsection{The coordinate maps \texorpdfstring{$\Phi(w)$}{P}} The coordinate maps are important for defining the Hilbert space $\cE$, the operator tuple $Y$, and the representing vector $\gamma.$ 

For each $w\in \la x\ra_{d+1}$ define $\Phi(w):\cH\to \cM$ by 
\[
\Phi(w)\xi\ =\ \big[(\delta_{u,w}\xi)_{u\in \la x\ra_{d+1}}\,\big], 
\]
where $\delta_{u,w}$ denotes the Kronecker delta. Thus, for all $v,w\in \la x\ra_{d+1}$ and $\xi,\eta\in \cH$,
\begin{equation}\label{eq:Phi-kernel}
\big\langle \Phi(w)\xi,\ \Phi(v)\eta\big\rangle_\cM
\ =\ \langle \xi,\ S_{v^ *w}\,\eta\rangle_\cH .
\end{equation}
		
\medskip
		
\subsection{The Hilbert space \texorpdfstring{$\cE$}{E} and the operator tuple \texorpdfstring{$Y$}{Y}}\label{ssec:E}
We define a Hilbert space $\cE$ and a self-adjoint tuple of operators $Y$ on $\cE.$
		
Let
\[ \cE\ :=\ {\rm span}\{\Phi(w)\xi:\ w\in \II{d},\ \xi\in \cH \}\ \subset \cM .
\]

For $i=1,\dots,g$, define $L_{x_{i}}: \cE\to \cM$ by
\begin{equation}\label{eq:Li}
L_{x_{i}}\big(\Phi(w)\xi\big)\ =\ \Phi(x_i w)\,\xi, \qquad(w\in \II{d},\ \xi\in \cH).
\end{equation}

\subsubsection{The \texorpdfstring{$L_{x_{j}}$}{} are well-defined}
If $z=\sum_{w\in \II{d}}\Phi(w)\xi_w\in \cE$ represents $0$ in $\cM$ (i.e.,\ $z\in\cN$), then for any $v\in \la x\ra_{d+1}$ and $\eta\in \cH$, using \eqref{eq:Phi-kernel} and $(x_iv)^ *=v^ *x_i$,
\begin{align*}
\big\langle L_{x_{i}} z,\ \Phi(v)\eta\big\rangle_\cM \ =&\  \sum_{w \in \la x\ra_{d}}\langle \xi_w,\ S_{v^ * x_i w}\eta\rangle_\cH 
\ =\   \sum_{w \in \la x \ra_{d}}\langle \xi_w,\ S_{(x_iv)^ * w}\eta\rangle_\cH \\
\ =&\  \big\langle z,\ \Phi(x_iv)\eta\big\rangle_\cM 
\ =\  0. 
\end{align*}
Since the vectors $\Phi(v)\eta$ span $\cM$, we have $L_{x_{i}} z\in \cN,$ proving  that $L_{x_{i}}$ is well-defined. 

Let $P_{\cE}$ be the orthogonal projection of $\cM$ onto $\cE.$ For $i=1,\dots, g,$ define $Y_{i}: \cE \to \cE$ by 
\[
Y_{i} = P_{\cE} L_{x_{i}}.
\]
\subsubsection{ The \texorpdfstring{$Y_j$}{Y} are self-adjoint} For $v,w\in \II{d}$ and $\xi,\eta\in \cH$,
\begin{align*}
\big\langle Y_{i}\Phi(w)\xi,\ \Phi(v)\eta\big\rangle_\cE 
\ =&\ \big\langle L_{x_{i}}\Phi(w)\xi,\ P_{\cE}\Phi(v)\eta\big\rangle_\cM \\
\ =&\ \big\langle L_{x_{i}}\Phi(w)\xi,\ \Phi(v)\eta\big\rangle_\cM \\
\ =&\ \langle \xi,\ S_{v^ *x_i w}\eta\rangle_\cH \\
\ =&\ \langle \xi,\ S_{(x_iv)^ * w}\eta\rangle_\cH \\
\ =&\ \big\langle \Phi(w)\xi,\ L_{x_{i}}\Phi(v)\eta\big\rangle_\cM \\
\ =&\  \big\langle Y_{i}^{*} \Phi(w)\xi,\ \Phi(v)\eta\big\rangle_\cE.
\end{align*}
		
\medskip
		
\subsection{The representing vector and evaluation}\label{ssec:vec} We construct the representing vector $\gamma$ and complete the proof. 
It is here that the vectorization map of Lemma~\ref{lem:vec identity} appears.

Define
\[
\gamma\ :=\ \mathrm{vec}\big(P_{\cE}\Phi(\varnothing)\big)\ \in\ \cH\otimes \cE, 
\]
where, as usual, $\varnothing$ is the empty word.
Consider a word $w=x_{i_1}\cdots x_{i_k}$ with $k\le d+1.$ We claim that 
\[
Y^{w}P_{\cE}\Phi(\varnothing)= P_{\cE}\Phi(w). 
\]
Indeed, for $\xi \in \cH$,
a word $|w|\le d$, and $1\le j\le \vg,$ 
\[ 
 Y_j P_{\cE} \Phi(w)\xi = P_{\cE} L_{x_j} \Phi(w)\xi = P_{\cE} \Phi(x_jw) \xi,
\]
since $\Phi(w)\xi\in \cE.$ Hence,
a finite induction argument gives,
\begin{align*} 
Y^{w} P_{\cE}\Phi(\varnothing) \xi 
\ =&\ Y^{w} \Phi(\varnothing) \xi \\
\ =&\  Y_{i_{1}} \dots Y_{i_{k}} \Phi(\varnothing) \xi \\
\ =&\  Y_{i_{1}} \dots Y_{i_{k-1}} P_{\cE} L_{x_{k}} \Phi(\varnothing) \xi\\
\ =&\  Y_{i_{1}} \dots Y_{i_{k-1}} \Phi(x_{k}) \xi\\
\ =&\  Y_{i_{1}} \Phi(x_{i_{2}} \cdots x_{i_{k}})\xi \\
\ =&\ P_{\cE} \Phi(w) \xi.
\end{align*}
		
Thus for $w \in \la x\ra_{d+1}$,
\begin{align} \label{eq:shift-gamma}
(I_\cH\otimes Y^{w})\,\gamma 
\ =&\  (I_\cH\otimes Y^{w}) \left( \sum\limits_{\delta} e_{\delta} \otimes P_{\cE} \Phi(\varnothing) e_{\delta}\right)  \nonumber \\
\ =&\  \sum\limits_{\delta} e_{\delta} \otimes P_{\cE}\Phi(w) e_{\delta}   \nonumber  \\
\ =&\  \mathrm{vec}\big(P_{\cE}\Phi(w)\big).
\end{align}

Let $p=\sum_{w\in \II{d+1}}P_w\,w$ and $q=\sum_{v\in \II{d}}Q_v\,v$. Using \eqref{eq:shift-gamma} and the standard vectorization identity (see Lemma~\ref{lem:vec identity})
\[
\big\langle (T\otimes I)\mathrm{vec}(A),\ (R\otimes I)\mathrm{vec}(B)\big\rangle =\operatorname{Tr}\!\big(TA^{*} B R^{*}\big), 
\]
we obtain
\begin{align*}
\big\langle p(Y)\gamma,\ q(Y)\gamma\big\rangle 
\ =&\  \sum_{w\in \II{d+1}} \sum\limits_{v \in \la x\ra_{d}}
\big\langle (P_w\otimes I)\,\mathrm{vec}(P_{\cE}\Phi(w)),\ (Q_v\otimes I)\,\mathrm{vec}(P_{\cE}\Phi(v))\big\rangle \\
\ =&\  \sum_{w\in \II{d+1}} \sum\limits_{v \in \la x\ra_{d}} \operatorname{Tr} \big(P_{w}  \Phi(w)^{*} P_{\cE}\Phi(v) Q_{v}^{*}\big) \\
\ =&\  \sum_{w\in \II{d+1}} \sum\limits_{v \in \la x\ra_{d}} \sum\limits_{\delta} \la  P_{\cE} \Phi(v) Q_{v}^{*}e_{\delta}, \Phi(w) P_{w}^{*} e_{\delta} \ra  \\
\ =&\  \sum_{w\in \II{d+1}} \sum\limits_{v \in \la x\ra_{d}} \sum\limits_{\delta} \la   \Phi(v) Q_{v}^{*}e_{\delta}, \Phi(w) P_{w}^{*} e_{\delta} \ra  \\
\ =&\  \sum_{w\in \II{d+1}} \sum\limits_{v \in \la x\ra_{d}} \sum\limits_{\delta} \la   Q_{v}^{*}e_{\delta}, S_{w^{*}v} P_{w}^{*} e_{\delta} \ra  \hspace{5mm} \text{(using \eqref{eq:Phi-kernel})}\\
\ =&\  \sum_{w\in \II{d+1}} \sum\limits_{v \in \la x\ra_{d}} \sum\limits_{\delta} \la  P_{w} S_{v^{*}w} Q_{v}^{*}e_{\delta}, e_{\delta} \ra  \\
\ =&\  \sum_{w\in \II{d+1}} \sum\limits_{v \in \la x\ra_{d}} \operatorname{Tr} \big(P_{w} S_{v^ *w}\,Q_{v}^{*}\big) \\
\ =&\ \sum_{w\in \II{d+1}} \sum\limits_{v \in \la x\ra_{d}} \operatorname{Tr} \big(S_{v^{*}w}Q_{v}^{*}P_{w} \big).
\end{align*}
By \eqref{eq:block-trace}, the right-hand side equals $\varphi(q^ *p)$, proving
\[
\varphi(q^ *p)\ =\ \langle p(Y)\gamma,\ q(Y)\gamma\rangle, 
\]
for all $p \in \cA_{d+1},\ q\in\mathcal A_{d}.$
\qed

\section{Proof of Theorem~\ref{thm:sos}}\label{sec:proof}

Suppose $f \in \cA_{2d}$ satisfies 
item~\ref{i:sos:ii} of Theorem~\ref{thm:sos}. That is, 
$f(X) \succeq 0$ for all ${\vg}$-tuples of self-adjoint matrices $X = (X_{1}, \dots, X_{\vg})$.	Our aim is to show that $f \in \cC_{d}.$ We will prove this statement via the contrapositive.  Accordingly, assume that $f \notin \cC_{d}.$ Since the top degree terms cannot cancel,  $f \notin \cC_{d+1}.$

Since $\cA_{2d+2}$ is a locally convex topological vector space and 
$\cC_{d+1}$ is closed by Proposition~\ref{prop:closedcone},  the Hahn--Banach separation theorem (see, e.g., \cite[Corollary 3.3.9]{D25}) implies that there exist a WOT continuous linear functional $\varphi : \cA_{2d+2} \to \mathbb{C},$ and real numbers $\gamma_{1}, \gamma_{2}$ such that 
\[
   {\real} (\varphi (f)) \ <\  \gamma_{1} \ <\ \gamma_{2} \ <\ {\real}(\varphi(p)) \quad \text{for all } p \in \cC_{d+1}.
\]
By \eqref{eq:coeff}, it follows that $f = f^{*}$ as $Z(f) = Z (f^{*}).$ Also $p = p^{*}$ for all $p \in \cC_{d+1}.$ Since $\cC_{d+1}$ is a cone, 
\[
   \varphi(f) \ =\  {\real}(\varphi(f)) \ <\ 0 \ \leq\ {\real}(\varphi(p)) \ =\  \varphi(p) \quad \text{for all } p \in \cC_{d+1}.
\]
Now apply Proposition~\ref{prop:GNS}. 
There exist a finite-dimensional Hilbert space $\cE$, a self-adjoint ${\vg}$-tuple $Y=(Y_1,\dots,Y_{\vg})$ on $\cE$, and a vector $\gamma\in \cH\otimes \cE$ such that
\[	\varphi(p)\ =\ \big\langle p(Y)\gamma,\ \gamma\big\rangle_{\cH\otimes \cE} \qquad\text{for all }p\in\cA_{2d}. 
\]
In particular,
\[
	0 > \varphi(f) = \langle f(Y)\gamma,\ \gamma \rangle .%
\]
Thus $f(Y)\not\succeq 0.$ 
\qed

\section{Proof of Theorem~\ref{thm:usos}}\label{sec:uproof}

The proof of Theorem~\ref{thm:usos} roughly follows the outline used in the proof of Theorem~\ref{thm:sos} in Section~\ref{sec:proof} above. We thus only explain the differences and adaptations needed to establish Theorem~\ref{thm:usos}.

Most of the notation we introduced for $\cA$ naturally carries over to $\ccA$. We will use $\ccC$ to denote the sum of squares in $\ccA$, and a straightforward adaptation of the results of Section~\ref{sec:sos} yields an analog of Corollary~\ref{cor:soscone} 
(and Remark~\ref{rem:soscone})
describing the cone 
$\ccC_d$ in $\ccA_{2d}$. 
Of course,
$V_d$ is now the Veronese column vector for (reduced) words $u\in\freeg$ with $|u|\leq d$.

\subsection{Creation operators and a tuple of unitaries}
The biggest change is in the construction of 
suitable operators out of the creation operators. That is,
we need to replace the self-adjoint $A_j$ of \eqref{eq:symcreate} with  unitary operators $U_j$. 
Let $\mathbf{F}_{d}$ denote the Hilbert space obtained as
 the span of (the orthonormal set) $(\freeg)_{d}$ of words of length at most $d.$  
 For notational purposes, let $\{x,x^{-1}\}$ denote the set   $\{x_1,\dots,x_{\vg},x_1^{-1},\dots,x_{\vg}^{-1}\}.$
Fix $y\in \{x,x^{-1}\}$ and let
\[
 \mathbf{M}_y \ =\  {\rm span} \big(  (\freeg)_{{d}-1}  \cup y  (\freeg)_{d-1} \big) \ \subseteq \ \mathbf{F}_{d}
\]
 Since if $w\in \mathbf{M}_{y^{-1}},$ then $yw\in \mathbf{M}_y,$ we obtain a linear map
\[
 L_y: \mathbf{M}_{y^{-1}} \to \mathbf{M}_{y},
 \qquad L_yw=yw.
\]   
 Given a reduced word $u\in (\freeg)_{d-1},$ 
 {$y^{-1}u \in y^{-1}(\freeg)_{d-1}$ (or $y^{-1}u\in (\freeg)_{d-2}$)
 and $L_y y^{-1} u =u.$ Thus $u$ is in the range of $L_y.$ Similarly, $ y u\in y (\freeg)_{d-1},$
 is in the range of $L_y$ since $u\in \mathbf{M}_{y^{-1}}.$ Hence $L_y$ is onto.} Since $\mathbf{M}_{y^{-1}}$ and $\mathbf{M}_y$ have the same dimension (by symmetry),
 $L_y$ is bijective.   {Now let $w,v\in (\freeg)_{d-1}  \cup y  (\freeg)_{d-1}$
 be given. Since $L_y$ is bijective 
   $L_y w=L_y v$
 if and only if $w=v$ and thus,
 \[
  \langle L_y w, L_y v \rangle \ =\   \langle w,v\rangle,
 \]
  for all $w,v\in (\freeg)_{d-1}  \cup y  (\freeg)_{d-1}.$  Since $(\freeg)_{d-1}  \cup y  (\freeg)_{d-1}$ is an
 orthonormal basis for $\mathbf{M}_{y^{-1}},$ it follows that
  $L_y$ is a unitary map.}  Because $\mathbf{M}_{y^{-1}}$ and $\mathbf{M}_y$ have the same dimension,
  they have the same codimension in $\mathbf{F}_{d}$ and hence $L_y$ extends to a unitary 
  operator on $\mathbf{F}_{d}$, called $U_y$.  Note that if $w\in (\freeg)_{d-1}$ and $z\in \{x,x^{-1}\},$
   then $L_z L_y  w = zy w.$ In particular,
  if $z=y^{-1},$ then $L_z L_y w=w.$   Finally, if $w\in (\freeg)_{d},$ then 
 \[
  L^w \varnothing = w.
 \]
  To prove  this claim, given $w\in (\freeg)_d,$ 
  write $w=y u$ where $u\in (\freeg)_{d-1}$ and $y\in \{x,x^{-1}\}.$
   Thus $L^w \varnothing = L_y L^u \varnothing = L_y u =yu.$

Next, 
the analog of Proposition~\ref{prop:bounded} in the unitary case is the following: If $p\in \ccC_d,$ then $\Gamma_p$ is norm bounded. More precisely, with the $U_y$
just constructed and 
 \[
  U=(U_{x_1}, \dots, U_{x_\vg}, U_{x_1^{-1}}, \dots,  U_{x_\vg^{-1}}), %
 \]
there exists a $\tau_{d}$ such that 
if $S\in \Gamma_p,$ then 
\[
 \|S\| \le \tau_{d}\; \|p(U)\|.
\]
To prove this claim, observe, for a word $w\in (\freeg)_d$ and vectors $\zeta,\eta\in \cH,$ 
\[
\langle p(U)  \zeta\otimes\varnothing, \eta\otimes w\rangle \ =\  \langle P_w \zeta,\eta\rangle.
\]
Hence, $\|P_w\|  \le \|p(U)\|.$  Now follow the rest of the proof of Proposition~\ref{prop:bounded} with the conclusion
$\|S_{v,w}\| \le N_{\rm red}(d)\; \|p(U)\|.$   

After the topology on $\ccA_{2d}$ has been defined via WOT-convergence of the coefficients as in Section~\ref{sec:top},
the proof of Proposition~\ref{prop:closedcone} translates essentially verbatim to show 
the closedness of the cone $\ccC_d$.

\subsection{A modified GNS construction}\label{ssec:uGNS}
The only other point that needs attention is the proof of a suitable GNS construction as in Proposition~\ref{prop:GNS}.
Since we cannot rely on non-cancellation of the highest order terms, we start with a continuous $\varphi:\ccA_{2d}\to\C$ that is positive on $\ccC_d;$ that is,
\[
\varphi(p^*p) \ \ge\  0
\]
for all $p\in\ccA_d.$  We go about obtaining $S$ and $\cE$ as in the self-adjoint case - with the obvious
adjustments.

For $y\in \{x,x^{-1}\},$ we now let 
\[
 \cM_y \ =\ {\rm span} \{\Phi(w)\xi : w\in \left (
(\freeg)_{d-1} \cup y (\freeg)_{d-1} \right ),  \, \xi\in \cH\} \ \subseteq\  \cE.
\]
 It is evident that if $w\in \cM_{y^{-1}},$ then $yw\in \cM_y.$ Hence we obtain a linear map
 $L_y :\cM_{y^{-1}} \to \cM_y$ by 
\[
 L_y \Phi(w)\xi \ =\  \Phi(y w)\xi.
\]

 That $L_y$ is well-defined works just as with the self-adjoint case.  That 
 $L_y$ is isometric  is a consequence of 
 \[
   S_{v,w} \ =\  S_{v^{-1}w} \ =\  S_{v^{-1} y^{-1} y w} \ =\  S_{(yv)^{-1} (yw)},
 \] 
  for the relevant $v,w$ and where our involution ${}^\ast$ satisfies $y^*=y^{-1}.$ 
  To see that $L_y$ is onto, observe if  $w\in \cM_{y},$ then $L_y L_{y^{-1}} w=w.$
  Hence $L_y$ is onto (and $L_{y^{-1}}$ is its inverse).
  Thus $L_y$ is unitary.

  Since $\cE$ is finite dimensional and  $L_y$ is bijective between them, the subspaces
   $\cM_{y^{-1}}$  and $\cM_y$ have the same codimension and thus $L_y$ extends to 
   a unitary map $U_y:\cE\to\cE.$  The novel ingredients  now in place, 
   following the, by now, beaten path of the cone separation GNS argument yields 
Theorem~\ref{thm:usos}. \qed

\section{Concluding remark and a problem}\label{sec:new}

The reader will have no difficulty verifying that both Theorem~\ref{thm:sos} and Theorem~\ref{thm:usos} remain true if we replace the algebra of coefficients $\cB(\cH)$ with a von Neumann algebra. The proofs carry over to this setting in a straightforward way.
 On the other hand, we do not know if the results still hold if the coefficients are from a $C^*$-algebra.

\section*{Declarations}

\subsection*{Competing interests} The authors have no relevant financial or non-financial interests to disclose.

\subsection*{Data availability} This article is theoretical in nature; no datasets were generated or analyzed. All results and proofs are contained within the article.

\makeatletter
\saved@setaddresses                    %
\makeatother

\newpage \printindex

\end{document}